\documentclass[12pt]{article}

\usepackage{amsmath} 
\usepackage{mathrsfs}
\usepackage{amssymb} 
\usepackage{theorem}
\usepackage{rotating}
\usepackage{stmaryrd}

\theoremstyle{change} 
\newtheorem{theorem}{Theorem}[section] 
\newtheorem{lemma}[theorem]{Lemma} 
\newtheorem{proposition}[theorem]{Proposition}

\theorembodyfont{\rmfamily} 
\newtheorem{remark}[theorem]{Remark}

\newtheorem{notation}[theorem]{Notation}

\newenvironment{proof}{\noindent{\bf Proof}\ }{\qed\bigskip}

\renewcommand{\le}{\leqslant}

\renewcommand{\marginpar}[1]{}

\newcommand{\alphabar}{\overline{\alpha}}

\newcommand{\Aut}{\mathrm{Aut}}

\newcommand{\BIGOP}[1]
  {\mathop{\mathchoice
  {\raise-0.22em\hbox{\huge $#1$}}
  {\raise-0.05em\hbox{\Large $#1$}}{\hbox{\large $#1$}}{#1}}}

\newcommand{\cdotG}{\cdot_G}
\newcommand{\cdotH}{\cdot_H}

\newcommand{\Gammatilde}{\widetilde{\Gamma}}

\newcommand{\id}{\mathrm{id}}

\newcommand{\im}{\mathrm{im}}

\newcommand{\Inn}{\mathrm{Inn}}

\newcommand{\myiso}{\buildrel\sim\over\to}

\newcommand{\lexp}[2]{\setbox0=\hbox{$#2$} \setbox1=\vbox to
                 \ht0{}\,\box1^{#1}\!#2}

\newcommand{\Out}{\mathrm{Out}}
\newcommand{\phibar}{\overline{\phi}}

\newcommand{\qed}{\nobreak\hfill
                   \vbox{\hrule\hbox{\vrule\hbox to 5pt
                   {\vbox to 8pt{\vfil}\hfil}\vrule}\hrule}}

\newcommand{\scrS}{\mathscr{S}}

\newcommand{\scrStilde}{\tilde{\scrS}}

\newcommand{\ZZ}{\mathbb{Z}}

\title{Orthogonal units of the bifree double Burnside ring\footnote{{\bf MR Subject Classification:}  19A22, 20C20.
{\bf Keywords:}  Burnside ring, double Burnside ring, unit group, blocks of finite groups.}}

\author{\small Robert Boltje\\
  \small Department of Mathematics\\
  \small University of California\\
  \small Santa Cruz, CA 95064\\
  \small U.S.A.\\
  \small boltje@ucsc.edu
  \and
  \small Philipp Perepelitsky\\
  \small Department of Mathematics\\ 
  \small University of California\\
  \small Santa Cruz, CA 95064\\
  \small U.S.A.\\
  \small pperepel@ucsc.edu}
\date{September 23, 2013\\ {\small (revised February 26, 2014)}}

\begin{document}
\sloppy


\maketitle


\begin{abstract}
The bifree double Burnside ring $B^\Delta(G,G)$ of a finite group $G$ has a natural anti-involution. We study the group $B^\Delta_\circ(G,G)$ of orthogonal units in $B^\Delta(G,G)$. It is shown that this group is always finite and contains a subgroup isomorphic to $B(G)^\times\rtimes \Out(G)$, where $B(G)^\times$ denotes the unit group of the Burnside ring of $G$ and $\Out(G)$ denotes the outer automorphism group of $G$. Moreover it is shown that if $G$ is nilpotent then $B^\Delta_\circ(G,G)\cong B(G)^\times\rtimes \Out(G)$. The results can be interpreted as positive answers to questions on equivalences of $p$-blocks of group algebras in the case that the block is the group algebra of a $p$-group.
\end{abstract}


\section{Introduction}\label{sec intro}
Let $G$ and $H$ be finite groups and let $F$ be an algebraically closed field of positive characteristic $p$. Moreover, let $A$ and $B$ be blocks of the group algebras $FG$ and $FH$, respectively. Our long term goal is to study necessary and sufficient conditions for the existence of an element $\gamma\in T^\Delta(A,B)$ with the property 
\begin{equation}\label{eqn inverse conditions 1}
  \gamma\cdotH\gamma^\circ = [A] \text{ in } T^\Delta(A,A) \quad\text{and}\quad 
  \gamma^\circ\cdotG\gamma = [B]\text{ in }T^  \Delta(B,B)\,. 
\end{equation}  
Here, $T^\Delta(A,B)$ denotes the Grothendieck group, with respect to direct sums, of the category of $p$-permutation $(A,B)$-bimodules (i.e., direct summands of finitely generated permutation modules when regarded as $F[G\times H]$-modules) whose indecomposable direct summands have a {\em twisted diagonal subgroup} of $G\times H$ as vertex, i.e., a subgroup of the form
\begin{equation}\label{eqn twisted diagonal}
  \Delta(R,\alpha,S)=\{(\alpha(s),s)\mid s\in S\}\,,
\end{equation}
for an isomorphism $\alpha\colon S\myiso R$ between subgroups $S\le H$ and $R\le G$. Moreover, the operation $-^\circ$ is induced by taking dual modules and $-\cdotH-$ is induced by taking the tensor product over $FH$.

\smallskip
We are interested in the following more specific questions: What invariants (for instance defect groups, fusion systems) of the blocks $A$ and $B$ coincide, given the existence of some $\gamma$ satisfying (\ref{eqn inverse conditions 1})? Does each $\gamma$ satisfying (\ref{eqn inverse conditions 1}) actually {\em determine} an isomorphism between defect groups and fusion systems? What is the group of auto-equivalences $\gamma\in T^\Delta(A,A)$ satisfying (\ref{eqn inverse conditions 1}) with $A=B$? Is it finite? These questions will be addressed in this generality in a forthcoming paper. The results in this paper form the first step on this path:

\smallskip
We consider the special case where $G$ and $H$ are $p$-groups. This forces $A=FG$ and $B=FH$. Moreover, the group $T^\Delta(A,B)$ is canonically isomorphic to $B^\Delta(G,H)$ the Grothendieck group of finite $(G,H)$-bisets, whose point stabilizers are twisted diagonal subgroups of $G\times H$, when regarded as left $G\times H$-sets.  The tensor product of modules becomes the tensor product of bisets and taking dual bimodules induces the map $B^\Delta(G,H)\to B^\Delta(H,G)$, $[G\times H/L]\mapsto [H\times G/L^\circ]$, where $L^\circ:=\{(y,x)\in H\times G\mid (x,y)\in L\}$. This motivates the study of elements $\gamma\in B^\Delta(G,H)$ with the property that 
\begin{equation}\label{eqn inverse conditions 2}
  \gamma\cdotH\gamma^\circ=[G] \text{ in } B^\Delta(G,G) \quad \text{and}\quad 
  \gamma^\circ\cdotG\gamma=[H] \text{ in } B^\Delta(H,H)
\end{equation}  
with the above questions in mind. We will see that they have positive answers.
 
 \smallskip
However, studying such elements makes sense for arbitrary finite groups $G$ and $H$ and we place ourselves into this more general situation. We define $B_\circ^\Delta(G,H)$ to be the set of all elements $\gamma\in B^\Delta(G,H)$ satisfying (\ref{eqn inverse conditions 2}). If $G=H$ then we call such an element an {\em orthogonal unit} of $B^\Delta(G,G)$. Note that $B_\circ^\Delta(G,G)$ is a group under $-\cdotG-$. We call an element $\gamma$ of $B^\Delta(G,H)$ {\em uniform} if there exists an isomorphism $\phi\colon H\myiso G$ such that $\gamma$ lies in the $\ZZ$-span of the standard basis elements $[(G\times H)/\Delta(\phi(V),\phi,V)]$, $V\le H$. We denote the outer automorphism group of a group $G$ by $\Out(G)$. It acts through ring automorphisms on the Burnside ring $B(G)$ by $\lexp{\phibar}{[G/R]}:=[G/\phi(R)]$, where $\phibar\in\Out(G)$ denotes the image of $\phi\in\Aut(G)$ and $R\le G$. Thus, $\Out(G)$ also acts on the unit group $B(G)^\times$ of $B(G)$ via group automorphisms.

\smallskip
Our main result is the following theorem.

\bigskip
\begin{theorem}\label{thm main}
Let $G$ and $H$ be finite groups.

\smallskip
{\rm (a)} Assume that $\gamma\in B^\Delta(G,H)$ satisfies $\gamma\cdotH\gamma^\circ = [G]$ in $B^{\Delta}(G,G)$. Then also $\gamma^\circ\cdotG\gamma = [H]$ in $B^\Delta(H,H)$. In particular $\gamma\in B^\Delta_\circ(G,H)$.

\smallskip
{\rm (b)} The set $B_\circ^\Delta(G,H)$ is finite.

\smallskip
{\rm (c)} The set $B_\circ^\Delta(G,H)$ is non-empty if and only if $G$ is isomorphic to $H$.

\smallskip
{\rm (d)} For each $\gamma\in B_\circ^\Delta(G,H)$ there exists an isomorphism $\phi\colon H\myiso G$, unique up to composition with inner automorphisms of $G$, such that $[(G\times H)/\Delta(G,\phi,H)]$ occurs in $\gamma$. Moreover, if $G$ is nilpotent then each $\gamma\in B_\circ^\Delta(G,H)$ is uniform.

\smallskip
{\rm (e)} If $G$ is nilpotent then the group $B_\circ^\Delta(G,G)$ is isomorphic to the semidirect product $B(G)^\times\rtimes\Out(G)$ with respect to the natural action of $\Out(G)$ on $B(G)^\times$.
\end{theorem}

We also show that, for any finite group $G,$ the group $B_\circ^\Delta(G,G)$ has a normal subgroup isomorphic to $B(G)^\times,$ (see Lemma 3.2(a)), but we have no control over the factor group.

\medskip
In \cite{Bc2} Bouc determined the group $B(G)^\times$ for a finite $p$-group $G$ in terms of explicit basis elements. Moreover if $G$ is a group of odd order then $B(G)^\times=\{\pm 1\}$ (see \cite[Proposition~6.7(iii)]{Y}). Thus, if $G$ is a finite nilpotent group of odd order then $B_\circ^\Delta(G,G)\cong \{\pm1\}\times\Out(G)$.

\medskip
The paper is arranged as follows: In Section~\ref{sec prel} we introduce some notation and recall basic facts about $G$-sets, $(G,H)$-bisets, as well as the Burnside ring $B(G)$ and double Burnside group $B(G,H)$. In Section~\ref{sec proof of Thm} we prove Theorem~\ref{thm main}. Finally, in Section~\ref{sec examples} we show that if $G$ is a Frobenius group with mild extra conditions then the group $B^\Delta_\circ(G,G)$ contains $B(G)^\times\rtimes \Out(G)$ as a proper subgroup and there exist elements in $B^\Delta_\circ(G,G)$ which are not uniform. We also compute $B^\Delta_\circ(G,G)$ explicitly for the alternating group $G=A_4$.

\bigskip
We are grateful to the referee for making a previous formulation of Proposition~\ref{prop frobenius} more conceptual and for suggesting the question in Remark~\ref{rem frobenius}(b).


\section{Preliminaries and cited results}\label{sec prel}

In this section we establish some notation and cite results about $G$-sets, $(G,H)$-bisets and their corresponding Grothendieck groups $B(G)$ and $B(G,H)$. The results, ideas and constructions (except Theorem 2.5 and possibly Lemma 2.4) provided in this section date back to earlier work of Adams, Gunawardena and Miller \cite{AGM}, Benson and Feshbach \cite{BF}, Bouc \cite{Bc3}, Martino and Priddy \cite{MP}, and Webb \cite {W}. We thank the referee of an earlier version of this paper for bringing some of these to our attention. For some of the statements in this section we provide quick proofs for the reader's convenience. For all statements without proof we refer the reader to \cite[Chapter~2]{Bc1}.

\begin{notation}\label{not beginning}
Let $G$, $H$, and $K$ be finite groups. 

(a) We indicate by $U\le G$ that $U$ is a subgroup of $G$. We write $U<G$ if $U\le G$ and $U\neq G$. For an element $x\in G$ we write $c_x\colon G\to G$, $g\mapsto xgx^{-1}$, for the inner automorphism induced by $x$. We also set $\lexp{x}{U}:=xUx^{-1}$ for $x\in G$ and $U\le G$. The set of subgroups of $G$ is denoted by $\scrS_G$. If $\scrS\subset \scrS_G$ is a subset that is closed under $G$-conjugation and under taking subgroups then $\scrStilde\subseteq\scrS$ will always denote a set of representatives of the conjugacy classes of subgroups in $\scrS$. We denote by $\scrS^\Delta_{G,H}\subseteq \scrS_{G\times H}$ the set of subgroups of $G\times H$ that are of the form
\begin{equation*}
   \Delta(R,\alpha,S):=\{(\alpha(h),h)\mid h\in S\}\,,
\end{equation*}
where $\alpha\colon S\myiso R$ is an isomorphism between a subgroup $S$ of $H$ and a subgroup $R$ of $G$. These subgroups are precisely the elements $L\in\scrS_{G\times H}$ that satisfy $(G\times\{1\})\cap L = \{(1,1)\}=(\{1\}\times H)\cap L$, and will be called {\em twisted diagonal} subgroups. The set $\scrS_{G,H}^\Delta\subseteq \scrS_{G\times H}$ is closed under $(G\times H)$-conjugation and under taking subgroups. 

\smallskip
(b) The Burnside ring of $G$, denoted by $B(G)$, is the Grothendieck ring of the category of finite left $G$-sets with respect to disjoint unions and direct products. Each finite left $G$-set $X$ gives rise to an element $[X]\in B(G)$. The elements $[G/R]$, $R\in \scrStilde_G$, form a $\ZZ$-basis of $B(G)$, the {\em standard basis}. For any subgroup $U$ of $G$, one has a ring homomorphism $\Phi_U\colon B(G)\to \ZZ$, determined by $[X]\mapsto |X^U|$, where $|X^U|$ denotes the number of $U$-fixed points of the $G$-set $X$. The collection of these maps yields an injective ring homomorphism $\Phi\colon B(G)\to \prod_{U\le G} \ZZ$, $[X]\mapsto (|X^U|)_{U\le G}$. Note that $\Phi_U=\Phi_{xUx^{-1}}$ for all $U\le G$ and $x\in G$. Moreover, for $S\le G$ one has
\begin{equation}\label{eqn mark}
  \Phi_U([G/S])=|\{gS\in G/S \mid U\le \lexp{g}{S}\}|\,.
\end{equation}

\smallskip
(c) The double Burnside group $B(G,H)$ is defined as the Grothendieck group of finite $(G,H)$-bisets $X$, i.e., finite sets equipped with a left action of $G$ and a commuting right action of $H$. A $(G,H)$-biset is called {\em bifree} if it is free when considered as left $G$-set and as right $H$-set. The Grothendieck group $B^\Delta(G,H)$ of bifree $(G,H)$-bisets can be considered as a subgroup of $B(G,H)$. We identify $(G,H)$-bisets with left $(G\times H)$-sets via the definitions $(g,h)\cdot x:=g\cdot x\cdot h^{-1}$ and $g\cdot x\cdot h:=(g,h^{-1})\cdot x$, for $(g,h)\in G\times H$ and $x\in X$. This way we will identify $B(G,H)$ with $B(G\times H)$ and $B^\Delta(G,H)$ with the $\ZZ$-span of the standard basis elements  $[(G\times H)/L]$, $L\in\scrStilde_{G,H}^\Delta$, of $B(G\times H)$.

\smallskip
(d) If $X$ is a $(G,H)$-biset then the set $X$ can be regarded as $(H,G)$-biset via $hxg:=g^{-1}xh^{-1}$, we denote this $(H,G)$-biset by $X^\circ$. This operation induces an isomorphism $-^\circ\colon B(G,H)\to B(H,G)$, $[(G\times H)/L]\to [(H\times G)/L^\circ]$, where $L^\circ:=\{(h,g)\mid (g,h)\in L\}$ for $L\le G\times H$. One has $B^\Delta(G,H)^\circ=B^\Delta(H,G)$ and $\Delta(R,\alpha,S)^\circ=\Delta(S,\alpha^{-1},R)$ for $\Delta(R,\alpha,S)\in\scrS_{G,H}^\Delta$.

\smallskip
(e) If $X$ is a $(G,H)$-biset and  $Y$ is an $(H,K)$-biset then the {\em tensor product} $X\times_H Y$ is defined as the set of $H$-orbits of $X\times Y$ under the $H$-action $h\cdot(x,y):=(xh^{-1},hy)$ for $h\in H$ and $(x,y)\in X\times Y$. The $H$-orbit of $(x,y)$ is denoted by $x\times_H y$ and the set $X\times_HY$ is a $(G,K)$-biset under $g(x\times_H y)k:=(gx)\times_H(yk)$. This operation defines a bilinear map 
\begin{equation*}
  -\cdotH-\colon B(G,H)\times B(H,K) \to B(G,K)\,,\quad ([X],[Y])\mapsto [X\times_H Y]\,,
\end{equation*}
which maps $B^\Delta(G,H)\times B^\Delta(H,K)$ to $B^\Delta(G,K)$ and satisfies $(a\cdotH b)^\circ = b^\circ\cdotH a^\circ$ for $a\in B(G,H)$ and $b\in B(H,K)$. If $G=H=K$ then the tensor product induces a ring structure on $B(G,G)$, the {\em double Burnside ring} of $G$ with identity element $[G]=[(G\times G)/\Delta(G)]$, where $G$ is viewed as $(G,G)$-biset via left and right multiplication, and $\Delta(R):=\Delta(R,\id_R,R)$, for $R\le G$. Subgroups of $G\times G$ of the form $\Delta(R)$ will be called {\em diagonal} subgroups. For each $(G,H)$-biset $X$ one has $[G]\cdotG[X] = [X] = [X]\cdotH [H]$ in $B(G,H)$.

\smallskip
{\rm (f)} For subgroups $L\le G\times H$ and $M\le H\times K$ one defines their composition (as relations) by
\begin{equation*}
  L*M:=\{(g,k)\in G\times K\mid \exists h\in H \text{ with } (g,h)\in L\text{ and }(h,k)\in M\}\,.
\end{equation*}
Note that $L*M\le G\times H$. Moreover, if $L=\Delta(R,\alpha,S)$ and $M=\Delta(S,\beta,T)$ then $L*M=\Delta(R,\alpha\beta,T)$.
\end{notation}

We say that a standard basis element $[G/R]$ of $B(G)$ occurs in an element $a\in B(G)$ if its coefficient is non-zero.

\begin{lemma}\label{lem mark properties}
Let $G$ and $H$ be finite groups, let $L\le G\times H$, and let $a\in B(G,H)$.

\smallskip
{\rm (a)} One has $\Phi_L(a) = \Phi_{L^\circ}(a^\circ)$.

\smallskip
{\rm (b)} If $a\in B^\Delta(G,H)$ and $L=\Delta(R,\alpha,S)\in\scrS_{G,H}^\Delta$ then the integer $\Phi_L(a)$ is divisible by $|C_G(R)|$ and by $|C_H(S)|$. 

\smallskip
{\rm (c)} The map $\Phi\colon B^\Delta(G,H)\to\prod_{L\in\scrS^\Delta_{G,H}} \ZZ$, $a\mapsto (\Phi_L(a))$, is injective. Its image is a subgroup of finite index in $(\prod_{L\in\scrS^\Delta_{G,H}}\ZZ)^{G\times H}$, the $(G\times H)$-fixed points under the natural action induced by the conjugation action of $G\times H$ on the indexing set $\scrS^\Delta_{G\times H}$, i.e., tuples that are constant on conjugacy classes.

\smallskip
{\rm (d)} Let $a\in B^\Delta(G,H)$ and $L\in\scrS^\Delta_{G,H}$. Then $L$ is maximal in $\scrS^\Delta_{G,H}$ with the property that $[(G\times H)/L]$ occurs in $a$ if and only if $L$ is maximal in $\scrS^\Delta_{G,H}$ with the property that $\Phi_L(a)\neq 0$.
\end{lemma}

\begin{proof}
(a) If  $X$ is a $(G,H)$-biset then the subsets $(X^\circ)^{L^\circ}$ and $X^L$ of $X$ coincide.

\smallskip
(b) Let $X$ be a bifree $(G,H)$-biset and view it as left $(G\times H)$-set. Then $X^L$ is naturally an $N_{G\times H}(L)$-set. Since $C_G(R)\times C_H(S)\le N_{G\times H}(L)$, $X^L$ can be regarded as $(C_G(R),C_H(S))$-biset. Since $X$ is bifree as $(G,H)$-biset, $X^L$ is bifree as $(C_G(R),C_H(S))$-biset. The result follows.

\smallskip
(c) This follows from Equation~(\ref{eqn mark}), the injectivity of $\Phi$, from $\Phi_{\lexp{(g,h)}{L}}=\Phi_L$ for $L\in\scrS_{G,H}^\Delta$ and $(g,h)\in G\times H$, and from counting ranks.

\smallskip
(d) By Equation~(\ref{eqn mark}), for $L,L'\in\scrS^\Delta_{G,H}$, one has $\Phi_L([(G\times H)/L'])\neq 0$ if and only if $L$ is conjugate to a subgroup of $L'$. The result follows.
\end{proof}

For finite groups $G$ and $H$ we denote the projections $G\times H\to G$ and $G\times H\to H$ by $p_1$ and $p_2$, respectively. The following result describes the decomposition of the tensor product of two transitive bisets, see~\cite[Lemma~2.3.24]{Bc1}.

\begin{lemma}\label{lem Mackey formula}
Let $G$, $H$, and $K$ be finite groups and let $L\le G\times H$ and $M\le H\times K$. Then
\begin{equation*}
  [(G\times H)/L]\cdotH[(H\times K)/M] = \sum_{h\in p_2(L)\backslash H/p_1(M)} 
  [(G\times K)/ (L*\lexp{(h,1)}{M})]
\end{equation*}
in $B(G,K)$, where $h$ runs through a set of representatives of the double cosets $p_2(L)\backslash H/p_1(M)$.
\end{lemma}

The results in \cite[2.5.5--2.5.8]{Bc1} imply that the map
\begin{equation}\label{eqn iota}
  \iota\colon B(G)\to B^\Delta(G,G)\,,\quad [G/R]\mapsto [(G\times G)/\Delta(R)]\,,
\end{equation}
is an injective ring homomorphism. (Note that, for $R\le G$, one has $\widetilde{G/R}\cong (G\times G)/\Delta(R)$ as $(G\times G)$-sets, using the notation from $\cite[2.5.6]{Bc1}$.)

\begin{lemma}\label{lem diag mark formula}
Let $R\le G$ and let $x\in B(G)$. Then $\Phi_{\Delta(R)}(\iota(x)) = |C_G(R)| \cdot \Phi_R(x)$.
\end{lemma}

\begin{proof}
It suffices to show this for $x=[G/S]$, $S\le G$. It is straightforward to check that if $T\subseteq G$ is a set of representatives for $G/S$ then $T\times G\subseteq G\times G$ is a set of representatives for $(G\times G)/\Delta(S)$. By Equation~(\ref{eqn mark}), the number $\Phi_{\Delta(R)}([(G\times G)/\Delta(S)]) = ((G\times G)/\Delta(S))^{\Delta(R)}$ is equal to the number of elements $(t,g)\in T\times G$ satisfying $\Delta(R)\le \lexp{(t,g)}{\Delta(S)} = \Delta(\lexp{t}{S},c_{tg^{-1}},\lexp{g}{S})$. But $\Delta(R)\le \Delta(\lexp{t}{S},c_{tg^{-1}},\lexp{g}{S})$ if and only if $R\le \lexp{t}{S}$ and $tg^{-1}\in C_G(R)$. This in turn is equivalent to $R\le\lexp{t}{S}$ and $g\in C_G(R)t$. This implies the result, since $\Phi_R([G/S])$ is equal to the number of elements $t\in T$ with $R\le \lexp{t}{S}$.
\end{proof}

Let $G$, $H$, and $K$ be finite groups and let $\rho\colon T\myiso R$ be an isomorphism between a subgroup $T$ of $K$ and a subgroup $R$ of $G$. We denote by $\Gamma_H(R,\rho,T)$ the set of triples $(\sigma,S,\tau)$, where $S\le H$ and $\sigma\colon S\myiso R$ and $\tau\colon T\myiso S$ are isomorphisms such that $\rho=\sigma\circ\tau$. Note that $H$ acts on $\Gamma_H(R,\rho,T)$ by $\lexp{h}{(\sigma, S,\tau)}= (\sigma c_{h^{-1}}, \lexp{h}{S},c_h\tau)$, for $h\in H$ and $(\sigma,S,\tau)\in\Gamma_H(R,\rho,T)$. We denote by $\Gammatilde_H(R,\rho,T)\subseteq \Gamma_H(R,\rho,T)$ a set of representatives of the $H$-orbits of $\Gamma_H(R,\rho,T)$. The following theorem is a slight reformulation of \cite[Theorem~2.5]{BD1}.

\begin{theorem}\label{thm BD}
Let $G$, $H$, and $K$ be finite groups, let $R\le G$ and $T\le K$ be isomorphic subgroups and let $\rho\colon T\myiso R$ be an isomorphism. Then, for any $\gamma\in B^\Delta(G,H)$ and $\delta\in B^\Delta(H,K)$, one has
\begin{equation*}
  \Phi_{\Delta(R,\rho,T)}(\gamma\cdotH\delta) 
  =  \sum_{(\sigma,S,\tau)\in\Gammatilde_H(R,\rho,T)} 
  |C_H(S)|^{-1} \cdot \Phi_{\Delta(R,\sigma,S)}(\gamma) \cdot \Phi_{\Delta(S,\tau,T)}(\delta)\,.
\end{equation*}
\end{theorem}

The following lemma is a consequence of Proposition~1.7(a) and (c) in \cite{BD1}.

\begin{lemma}\label{lem normalizer order}
Let $G$ and $H$ be finite groups, let $\Delta(R,\alpha,S)\in\scrS_{G,H}^\Delta$, and set
\begin{equation*}
  N_\alpha:=\{h\in N_H(S) \mid \exists g\in N_G(R)\colon \alpha c_h = c_g \alpha \text{ on $S$}\}
\end{equation*}
and
\begin{equation*}
  N_{\alpha^{-1}}:= \{ g\in N_G(R) \mid \exists h\in N_H(S)\colon \alpha c_h = c_g\alpha \text{ on $S$}\}\,.
\end{equation*}
Then
\begin{equation*}
  |N_{G\times H}(\Delta(R,\alpha,S))| = |N_{\alpha^{-1}}|\cdot |C_H(S)| = |C_G(R)| \cdot |N_\alpha|\,.
\end{equation*}
\end{lemma}

\begin{lemma}\label{lem isomorphisms}
Let $G$, $H$, and $K$ be finite groups and let $\phi\colon H\myiso G$ and $\psi\colon K\myiso H$ be isomorphisms.

\smallskip
{\rm (a)} One has
\begin{equation*}
  [(G\times H)/\Delta(G,\phi,H)] \cdotH [(H\times K)/\Delta(H,\psi,K)] = [(G\times K)/\Delta(G,\phi\psi,K)]\,.
\end{equation*}

\smallskip
{\rm (b)} Let $\delta:=[(G\times H)/\Delta(G,\phi,H)]\in B^\Delta(G,H)$. The map
\begin{equation*}
  \delta\cdotH - \colon B^\Delta(H,H)\to B^\Delta(G,H)\,, \gamma\mapsto\delta\cdotH\gamma\,,
\end{equation*}
is a group isomorphism with inverse $\delta^\circ\cdotG -$.
\end{lemma}

\begin{proof}
Part~(a) follows immediately from the formula in Lemma~\ref{lem Mackey formula} and Part~(b) follows from Part~(a) and the last statement in \ref{not beginning}(e).
\end{proof}

Lemma~\ref{lem isomorphisms}(a) immediately implies that the map
\begin{equation}\label{eqn eta}
  \eta\colon \ZZ\Out(G)\to B^\Delta(G,G)\,,\quad \phibar\mapsto[(G\times G)/\Delta(G,\phi,G)]\,,
\end{equation}
is a ring homomorphism. Moreover, Lemma~\ref{lem isomorphisms}(a) and Lemma~\ref{lem Mackey formula} imply that the map
\begin{equation}\label{eqn rho}
  \rho\colon B^\Delta(G,G)\to \ZZ\Out(G)\,,\quad [(G\times G)/\Delta(R,\alpha,S)] \mapsto
       \begin{cases} \alphabar, &\text{if $R=S=G$,}\\ 0, &\text{otherwise,} \end{cases}
\end{equation}
is a ring homomorphism. Clearly one has $\rho\circ\eta = \id_{\ZZ\Out(G)}$.

\bigskip
The first part of the following lemma is well-known and follows immediately from the injectivity of the ring homomorphism $\Phi\colon B(G)\to\prod_{R\le G} \ZZ$. The second part is a result of Yoshida, see \cite[Proposition~6.7(iii)]{Y}.

\begin{lemma}\label{lem units}
Let $G$ be a finite group.

\smallskip
{\rm (a)} The unit group $B(G)^\times$ is a finite elementary abelian $2$-group.

\smallskip
{\rm (b)} If $G$ has odd order then $B(G)^\times=\{\pm 1\}$.
\end{lemma}


\section{Proof of Theorem~\ref{thm main}}\label{sec proof of Thm}

The goal of this section is the proof of Theorem~\ref{thm main}. We will need several lemmas.

\begin{lemma}\label{lem 1}
Let $G$ and $H$ be finite groups and assume that $\gamma\in B^\Delta(G,H)$ satisfies $\gamma\cdotH \gamma^\circ=[G]$ in $B^\Delta(G,G)$.

\smallskip
{\rm (a)} For each $R\le G$ there exists a twisted diagonal subgroup $\Delta(R,\alpha,S)\le G\times H$, unique up to $(1\times H)$-conjugacy, such that $\Phi_{\Delta(R,\alpha,S)}(\gamma)\neq 0$. Moreover, for any such subgroup one has 
\begin{equation*}
  |C_G(R)| = |\Phi_{\Delta(R,\alpha,S)}(\gamma)| = |C_H(S)|\quad \text{and} \quad N_{\alpha^{-1}}=N_G(R)\,.
\end{equation*}

\smallskip
{\rm (b)} There exists an isomorphism $\phi\colon H\myiso G$, unique up to composition with inner automorphisms of $G$, such that $\Phi_{\Delta(G,\phi,H)}(\gamma)\neq 0$.
\end{lemma}

\begin{proof}
(a) By Lemma~\ref{lem diag mark formula}, Theorem~\ref{thm BD}, and Lemma~\ref{lem mark properties}(a), we have
\begin{align*}
  |C_G(R)| & = \Phi_{\Delta(R)}([(G\times G)/\Delta(G)]) = \Phi_{\Delta(R)}(\gamma\cdotH\gamma^\circ) \\
  & = \sum_{(\alpha,S,\alpha^{-1})\in\Gammatilde_H(R,\id_R,R)} |C_H(S)|^{-1}\cdot \Phi_{\Delta(R,\alpha,S)}(\gamma) 
      \cdot \Phi_{\Delta(S,\alpha^{-1},R)}(\gamma^\circ) \\
  & = \sum_{(\alpha,S,\alpha^{-1})\in\Gammatilde_H(R,\id_R,R)} |C_H(S)|^{-1}\cdot 
        \Phi_{\Delta(R,\alpha,S)}(\gamma)^2\,.
\end{align*}
This implies
\begin{equation*}
   1=\sum_{(\alpha,S,\alpha^{-1})\in\Gammatilde_H(R,\id_R,R)} \frac{|\Phi_{\Delta(R,\alpha,S)}(\gamma)|}{|C_G(R)|}
   \cdot\frac{|\Phi_{\Delta(R,\alpha,S)}(\gamma)|}{|C_H(S)|}\,\,.
\end{equation*}
By Lemma~\ref{lem mark properties}(b), for each $(\alpha,S,\alpha^{-1})\in\Gammatilde_H{(R,\id_R,R)}$, the integer $\Phi_{\Delta(R,\alpha,S)}(\gamma)$ is divisible by $|C_G(R)|$ and by $|C_H(S)|$. This implies that there exists a unique element $(\alpha,S,\alpha^{-1})\in\Gammatilde_H(R,\id_R,R)$ such that $\Phi_{\Delta(R,\alpha,S)}(\gamma)\neq 0$. For this element $(\alpha,S,\alpha^{-1})$ the above equation further implies $|C_G(R)|=|\Phi_{\Delta(R,\alpha,S)}(\gamma)|=|C_H(S)|$. If also $\Delta(R,\beta,T)$ is a twisted diagonal subgroup of $G\times H$ such that $\Phi_{\Delta(R,\beta,T)}(\gamma)\neq 0$, then, by the above uniqueness, the triple $(\beta,T,\beta^{-1})\in \Gamma_H(R,\id_R,R)$ is $H$-conjugate to $(\alpha,S,\alpha^{-1})$. This implies that $\Delta(R,\beta,T)$ is $(1\times H)$-conjugate to $\Delta(R,\alpha, S)$.

To see that $N_{\alpha^{-1}}=N_G(R)$, let $g\in N_G(R)$. Then $\Phi_{\Delta(R,c_g\alpha,S)}(\gamma) = \Phi_{\Delta(R, \alpha,S)}(\gamma)\neq 0$ and hence $\Delta(R,c_g\alpha,S)$ is $(1\times H)$-conjugate to $\Delta(R,\alpha, S)$. Thus, there exists $h\in H$ such that $\Delta(R,c_g\alpha c_h^{-1},\lexp{h}{S})=\Delta(R,\alpha,S)$. This implies $g\in N_{\alpha^{-1}}$.

\smallskip
(b) By Part~(a), applied to $R=G$, there exists a twisted diagonal subgroup of $G\times H$, unique up to $(1\times H)$-conjugation, of the form $\Delta(G,\phi,S)$ with $\Phi_{\Delta(G,\phi,S)}(\gamma)\neq 0$. Moreover, Part~(a) applied to $R=\{1\}$ yields $|G|=|C_G(1)|=|C_H(1)| = |H|$. This implies that $S=H$ and the uniqueness part in (a) completes the proof.
\end{proof}

\bigskip
{\bf Proof} {\em of Theorem~\ref{thm main}(a):}\quad By Lemma~\ref{lem 1}(b), there exists an isomorphism $\phi\colon H\myiso G$. By Lemma~\ref{lem isomorphisms}(b), the free abelian groups $B^\Delta(H,H)$ and $B^\Delta(G,H)$ have the same rank. Consider the additive group homomorphism
\begin{equation*}
  \gamma\cdotH - \colon B^\Delta(H,H)\to B^\Delta(G,H)\,.
\end{equation*}
Since $\gamma\cdotH\gamma^\circ=[G]$, this map is (split) surjective, and considering ranks it is also injective. This and the equation
\begin{equation*}
  \gamma\cdotH(\gamma^\circ\cdotG\gamma) = (\gamma\cdotH\gamma^\circ)\cdotG\gamma = [G]\cdotG\gamma 
  = \gamma = \gamma\cdotH [H]
\end{equation*}
now imply $\gamma^\circ\cdotG\gamma=[H]$.
\qed

\bigskip
{\bf Proof} {\em of Theorem~\ref{thm main}(b):}\quad Let $\gamma\in B_\circ^\Delta(G,H)$. Lemma~\ref{lem 1}(a) implies that $\Phi_{\Delta(R,\alpha,S)}(\gamma)\in\{\pm|C_G(R)|,0\}$ for all twisted diagonal subgroups $\Delta(R,\alpha,S)$ of $G\times H$. Moreover, if $L\le G\times H$ is not twisted diagonal then $\Phi_L(\gamma)=0$ by Equation~(\ref{eqn mark}). Thus, $\Phi(\gamma)$ belongs to a finite subset of $\prod_{L\in\scrS_{G\times H}} \ZZ$. Since $\Phi$ is injective, $\gamma$ belongs to a finite subset of $B^\Delta(G,H)$.
\qed

\bigskip
{\bf Proof} {\em of Theorem~\ref{thm main}(c):}\quad If $G$ and $H$ are isomorphic and if $\phi\colon H\myiso G$ is an isomorphism then $[(G\times H)/\Delta(G,\phi,H)]$ belongs to $B^\Delta_\circ(G,H)$, by Lemma~\ref{lem isomorphisms}(a).

Conversely, if $\gamma\in B^\Delta_\circ(G,H)$ then Lemma~\ref{lem 1}(b) implies that $G\cong H$.
\qed

\bigskip
{\bf Proof} {\em of the first part of Theorem~\ref{thm main}(d):} \quad By Lemma~\ref{lem 1}(b), there exists an isomorphism $\phi\colon H\myiso G$ such that $\Phi_{\Delta(G,\phi,H)}(\gamma)\neq 0$. Since $\Delta(G,\phi,H)$ is maximal among all twisted diagonal subgroups, the standard basis element $[(G\times H)/\Delta(G,\phi,H)]$ occurs in $\gamma$ with non-zero coefficient, see Lemma~\ref{lem mark properties}(d). Assume that also $[(G\times H)/\Delta(G,\psi,H)]$ occurs in $\gamma$ with non-zero coefficient for some isomorphism $\psi\colon H\to G$. Again, since $\Delta(G,\psi,H)$ is a maximal twisted diagonal subgroup, this implies $\Phi_{\Delta(G,\psi,H)}(\gamma)\neq 0$. Now Lemma~\ref{lem 1}(b) implies the uniqueness statement in Theorem~\ref{thm main}(d).
\qed

\bigskip
The next lemma collects some results about the relationship between the groups $B_\circ^\Delta(G,G), B(G)^\times,$ and $\Out(G).$ Except for the normality statement of Part~(a), they are likely well-known to experts. We include proofs for all statements for the reader's convenience. Note that Parts~(a) and (b) produce elements in $B_\circ^\Delta(G,G)$ in a natural way.

\begin{lemma}\label{lem 2}
Let $G$ be a finite group.

\smallskip
{\rm (a)} One has $\iota(B(G)^\times)=\iota(B(G))\cap B_\circ^\Delta(G,G)$, and $\iota(B(G)^\times)$ is a normal subgroup of $B_\circ^\Delta(G,G)$. Here, $\iota\colon B(G)\to B^\Delta(G,G)$ denotes the ring homomorphism from (\ref{eqn iota}).

\smallskip
{\rm (b)} The ring homomorphism $\eta\colon \ZZ\Out(G)\to B^\Delta(G,G)$ from (\ref{eqn eta}) restricts to an injective group homomorphism
\begin{equation*}
  \eta\colon \Out(G)\to B_\circ^\Delta(G,G)\,.
\end{equation*}

\smallskip
{\rm (c)} For each $\gamma\in B_\circ^\Delta(G,G)$ there exists a unique $\phibar\in\Out(G)$ and a unique $\epsilon\in\{\pm 1\}$ such that $\rho(\gamma)=\epsilon\cdot\phibar$ (see (\ref{eqn rho}) for the definition of $\rho$). Moreover, the resulting map
\begin{equation*}
  \pi\colon B_\circ^\Delta(G,G)\to\Out(G)\,,\quad \gamma\mapsto \phibar\,,
\end{equation*}
is a surjective group homomorphism with $\pi\circ\eta=\id_{\Out(G)}$.

\smallskip
{\rm (d)} Set $\Lambda_G:=\ker(\pi)\trianglelefteq B_\circ^\Delta(G,G)$ and $\Delta_G:=\im(\eta)\le B_\circ^\Delta(G,G)$. Then $\iota(B(G)^\times)\le \Lambda_G$ and $B_\circ^\Delta(G,G) = \Lambda_G\rtimes\Delta_G$ is a semidirect product.

\smallskip
{\rm (e)} For $\phibar\in\Out(G)$ and $u\in B(G)^\times$ one has $\eta(\phibar)\cdotG \iota(u) \cdotG \eta(\phibar)^{-1} = \iota(\lexp{\phibar}{u})$.
\end{lemma}

\begin{proof}
(a) Recall from (\ref{eqn iota}) that $\iota\colon B(G)^\times\to B^\Delta(G,G)^\times$ is an injective group homomorphism. We first show that $\iota(B(G)^\times)\subseteq B_\circ^\Delta(G,G)$. If $u\in B(G)^\times$ then $u^2=1$ (see Lemma~\ref{lem units}(a)) and $\iota(u)^\circ=\iota(u)$ (see~\ref{not beginning}(d)). Thus, $\iota(u)^\circ=\iota(u)=\iota(u)^{-1}$ and $\iota(u)\in B_\circ^\Delta(G,G)$. Conversely, assume that $u\in B(G)$ and $\iota(u)\in B_\circ^\Delta(G,G)$. Again, $\iota(u)^\circ=\iota(u)$ implies that $\iota(u)^2=1$. Since $\iota$ is an injective ring homomorphism, this yields $u^2=1$ in $B(G)$ and $u\in B(G)^\times$. This completes the proof of the first equation.

Next we show that $\iota(B(G)^\times)$ is normal in $B_\circ^\Delta(G,G)$. Let $u\in B(G)^\times$ and $\gamma\in B_\circ^\Delta(G,G)$. By the first equation it suffices to show that $\gamma\cdotG\iota(u)\cdotG\gamma^\circ\in \iota(B(G))$. Note that $\iota(B(G))$ is equal to the $\ZZ$-span of the standard basis elements $[(G\times G)/\Delta(R)]$ with $R\le G$. Thus, by Lemma~\ref{lem mark properties}(d), it suffices to show that $\Phi_{\Delta(R,\alpha,S)}(\gamma\cdotG\iota(u)\cdotG\gamma^\circ) =0$ for all $\Delta(R,\alpha,S)\in\scrS_{G,G}^\Delta$ which are not conjugate to the diagonal subgroup $\Delta(R)$, $R\le G$. By the uniqueness part of Lemma~\ref{lem 1}(a), it suffices to show that $\Phi_{\Delta(R)}(\gamma\cdotG \iota(u)\cdotG\gamma^\circ) \neq 0$ for all $R\le G$. By Theorem~\ref{thm BD} we have
\begin{equation*}
  \Phi_{\Delta(R)}(\gamma\cdotG\iota(u)\cdotG\gamma^\circ) = 
  \sum_{(\alpha, S, \alpha^{-1})\in\Gammatilde_G(R,\id_R,R)}
  |C_G(S)|^{-1}\cdot\Phi_{\Delta(R,\alpha,S)}(\gamma)\cdot
  \Phi_{\Delta(S,\alpha^{-1},R)}(\iota(u)\cdotG\gamma^\circ)\,.
\end{equation*}
Moreover, by Lemma~\ref{lem 1}(a), there exists a unique element $(\alpha,S,\alpha^{-1})\in\Gammatilde_G(R,\id_R,R)$ such that $\Phi_{\Delta(R,\alpha,S)}(\gamma)\neq 0$. Thus, it suffices to show that for this element $(\alpha,S,\alpha^{-1})$ one has $\Phi_{\Delta(S,\alpha^{-1},R)}(\iota(u)\cdotG\gamma^{\circ})\neq 0$. But, again by Theorem~\ref{thm BD}, one has
\begin{equation*}
  \Phi_{\Delta(S,\alpha^{-1},R)}(\iota(u)\cdotG\gamma^\circ) = 
  \sum_{(\sigma,T,\tau)\in\Gammatilde_G(S,\alpha^{-1},R)} 
  |C_G(T)|^{-1}\cdot \Phi_{\Delta(S,\sigma,T)}(\iota(u))\cdot
  \Phi_{\Delta(T,\tau,R)}(\gamma^\circ)\,.
\end{equation*}
Without loss of generality, we may assume that $(\id_S,S,\alpha^{-1})\in\Gammatilde_G(S,\alpha^{-1},R)$. Since $\Phi_{\Delta(S,\alpha^{-1},R)}(\gamma^\circ)=\Phi_{\Delta(R,\alpha,S)}(\gamma)\neq 0$ and since $\Phi_{\Delta(T,\tau,R)}(\gamma^\circ)=\Phi_{\Delta(R,\tau^{-1},T)}(\gamma)= 0$ for all $(\sigma,T,\tau)\in \Gammatilde_G(S,\alpha^{-1},R)$ different from $(\id_S,S,\alpha^{-1})$ (by Lemma~\ref{lem 1}(a)), it suffices to show that $\Phi_{\Delta(S)}(\iota(u))\neq 0$. But $\Phi_{\Delta(S)}(\iota(u))=|C_G(S)|\cdot\Phi_S(u)\in \{\pm|C_G(S)|\}$ by Lemma~\ref{lem diag mark formula} and since $\Phi_S(u)\in\{\pm1\}$. This completes the proof of (a).

\smallskip
(b) This is immediate from Lemma~\ref{lem isomorphisms}(a), since $\Delta(G,\phi,G)^\circ=\Delta(G,\phi^{-1},G)$.

\smallskip
(c) By Lemma~\ref{lem 1}(b), there exists a unique element $\phibar\in\Out(G)$ such that $\Phi_{\Delta(G,\phi,G)}(\gamma)\neq 0$. So, by Lemma~\ref{lem mark properties}(d), we have $\rho(\gamma)=\epsilon\cdot[(G\times G)/\Delta(G,\phi,G)]$ for some $\epsilon \in\ZZ$. By Lemma~\ref{lem 1}(a) and Lemma~\ref{lem normalizer order}, we obtain
\begin{equation*}
  |Z(G)| = 
  |\Phi_{\Delta(G,\phi,G)}(\gamma)| = |\epsilon| \cdot [N_{G\times G}(\Delta(G,\phi,G))\colon \Delta(G,\phi,G)] =
  |\epsilon| \cdot |Z(G)|.
\end{equation*}
This implies $\epsilon\in\{\pm 1\}$ and the result follows.

\smallskip
(d) and (e) are straightforward verifications, left to the reader.
\end{proof}

\begin{lemma}\label{lem 3}
Let $G$, $H$, and $K$ be finite groups, let $\psi\colon K\myiso H$ be an isomorphism, set $\delta:=[(H\times K)/\Delta(H,\psi,K)]\in B^\Delta(H,K)$, and let $\gamma\in B^\Delta(G,H)$. Then, $\gamma$ is uniform if and only if $\gamma\cdotH\delta$ is uniform.
\end{lemma}

\begin{proof}
If $\gamma$ is uniform with respect to the isomorphism $\phi\colon H\myiso G$ then Lemma~\ref{lem Mackey formula} shows that $\gamma\cdotH\delta$ is uniform with respect to the isomorphism $\phi\circ\psi\colon K\myiso G$. Using $\delta^\circ$, the same argument yields the converse.
\end{proof}

\begin{lemma}\label{lem 4}
For a finite group $G$ the following are equivalent:

\smallskip
{\rm (i)} Each $\gamma\in B_\circ^\Delta(G,G)$ is uniform.

\smallskip
{\rm (ii)} Each $\gamma\in \Lambda_G$ is uniform.

\smallskip
{\rm (iii)} $\iota(B(G)^\times) = \Lambda_G$.
\end{lemma}

\begin{proof}
Clearly, (i) implies (ii). Next we show that (ii) implies (iii). By Lemma~\ref{lem 2}(d), it suffices to show that $\Lambda_G\le \iota(B(G)^\times)$. By Lemma~\ref{lem 2}(a), it suffices to show that $\Lambda_G\subseteq \iota(B(G))$, the $\ZZ$-span of the elements $[(G\times G)/\Delta(R)]$, $R\le G$. But this is an immediate consequence of (ii). Finally, we show that (iii) implies (i). Let $\gamma\in B_\circ^\Delta(G,G)$. By Lemma~\ref{lem 2}(d), we can write $\gamma=\gamma'\cdotG\delta$ for some $\gamma'\in\Lambda_G$ and $\delta\in\Delta_G$. Since each element in $\iota(B(G)^\times)$ is uniform, (iii) implies that $\gamma'$ is uniform. Now Lemma~\ref{lem 3} implies that also $\gamma$ is uniform.
\end{proof}

\begin{lemma}\label{lem 5}
Let $G$ be a finite group, let $\gamma\in \Lambda_G$ and assume that $\gamma$ is not uniform. Then there exists a subgroup $\Delta(R,\alpha,S)\in\scrS_{G,G}^\Delta$ which is not $(G\times G)$-conjugate to a diagonal subgroup $\Delta(U)$, $U\le G$, and which satisfies $\Phi_{\Delta(R,\alpha,S)}(\gamma)\neq 0$. If $\Delta(R,\alpha,S)$ is maximal among these subgroups then the  coefficient of $[(G\times G)/\Delta(R,\alpha,S)]$ in $\gamma$ is equal to $\pm1$, $N_G(R)=R<G$ and $N_G(S)=S<G$.
\end{lemma}

\begin{proof}
Clearly, a twisted diagonal subgroup $\Delta(R,\alpha,S)$ as in the statement exists, by Lemma~\ref{lem mark properties}(d). Assume that $\Delta(R,\alpha,S)$ is maximal as in the statement of the lemma, and let $a\in\ZZ$ denote the coefficient of the standard basis element $[(G\times G)/\Delta(R,\alpha,S)]$ in $\gamma$. Note that $R<G$, since $\gamma\in\Lambda_G$ and $\Delta(R,\alpha,S)$ is not $(G\times G)$-conjugate to a diagonal subgroup of $G\times G$. By Lemma~\ref{lem mark properties}(d), the maximality property of $\Delta(R,\alpha,S)$ implies that $a\neq 0$. The maximality also implies
\begin{equation}\label{eqn one}
  \Phi_{\Delta(R,\alpha,S)}(\gamma)= a\cdot[N_{G\times G}(\Delta(R,\alpha,S))\colon \Delta(R,\alpha,S)] \neq 0\,.
\end{equation}
Further, Lemma~\ref{lem 1}(a) implies that
\begin{equation}\label{eqn two}
  |C_G(R)|=|\Phi_{\Delta(R,\alpha,S)}(\gamma)|=|C_G(S)|\quad\text{and}\quad N_{\alpha^{-1}}=N_G(R)\,.
\end{equation}
On the other hand, by Lemma~\ref{lem normalizer order}, we have
\begin{equation}\label{eqn three}
  |N_{G\times G}(\Delta(R,\alpha,S))| = |C_G(R)|\cdot|N_\alpha| = |C_G(S)|\cdot|N_{\alpha^{-1}}|\,.
\end{equation}
Since also $|R|=|\Delta(R,\alpha,S)|=|S|$, we obtain
\begin{equation*}
  |\Phi_{\Delta(R,\alpha,S)}(\gamma)|\cdot |R|\ 
  {\buildrel(\ref{eqn one})\over=} \ |a|\cdot|N_{G\times G}(\Delta(R,\alpha,S))| \
  {\buildrel(\ref{eqn three})\over =} \ |a|\cdot|C_G(S)|\cdot|N_{\alpha^{-1}}| \
  {\buildrel(\ref{eqn two})\over = } \ |\Phi_{\Delta(R,\alpha,S)}(\gamma)|\cdot|a|\cdot|N_G(R)|\,.
\end{equation*}
This implies $N_G(R)=R$ and $a\in\{\pm 1\}$. Since also $\gamma^\circ\in\Lambda_G$ is not uniform, we obtain by symmetry that $N_G(S)=S<G$.
\end{proof}

\bigskip
{\bf Proof} {\em of Theorem~\ref{thm main}(e):}\quad Assume that $G$ is nilpotent. By Lemma~\ref{lem 2}(d) it suffices to show that $\Lambda_G=\iota(B(G)^\times)$. Assume by contradiction that this does not hold. Then, by Lemma~\ref{lem 4}, there exists $\gamma\in\Lambda_G=\ker(\pi)$ which is not uniform. By Lemma~\ref{lem 5} there exists a proper subgroup $R<G$ with $N_G(R)=R$, in contradiction to $G$ being nilpotent, and the proof is complete. \qed

\bigskip
{\bf Proof} {\em of the last statement in Theorem~\ref{thm main}(d):} Assume that $G$ is nilpotent. Since $B_\circ^\Delta(G,H)$ is not empty, there exists an isomorphism $\psi\colon G\myiso H$, by Theorem~\ref{thm main}(c). Now the remaining statement of Theorem~\ref{thm main}(d) follows from Lemma~\ref{lem 3} and the uniformity of elements in $B_\circ^\Delta(G,G)$, which was proved in the previous paragraph. \qed

\bigskip
This completes the proof of the theorem.


\section{Examples}\label{sec examples}

The following proposition shows that in general, if $G$ is not nilpotent, the subgroup $\iota(B(G)^\times)$ can be properly contained in $\Lambda_G$. In other words, by Lemma~\ref{lem 4}, there can exist orthogonal units in $B^\Delta(G,G)$ that are not uniform.

\begin{proposition}\label{prop frobenius}
Let $G$ be a Frobenius group with Frobenius complement $H$. For each $\alpha\in\Aut(H)$, set
\begin{equation*}
  \gamma_\alpha:=[(G\times G)/\Delta(G)]-[(G\times G)/\Delta(H)]+[(G\times G)/\Delta(H,\alpha,H)]\,.
\end{equation*}
Then $\alpha\mapsto \gamma_\alpha$ induces an injective group homomorphism $j\colon \Out(H)\to \Lambda_G$ with $j(\Out(H))\cap \iota(B(G)^\times)=\{1\}$.
\end{proposition}

\begin{proof}
Clearly $\gamma_{\id_H}=1_{B(G,G)}$ and $\gamma_\alpha^\circ=\gamma_{\alpha^{-1}}$, for all $\alpha\in\Aut(H)$. Let $\alpha,\beta\in\Aut(H)$ and $k:=|H\backslash G/H|-1$. The formula in Lemma~\ref{lem Mackey formula} yields
\begin{equation*}
  [(G\times G)/\Delta(H,\alpha,H)] \cdotG [(G\times G)/\Delta(H,\beta,H)] = 
  [(G\times G)/\Delta(H,\alpha\beta,H)] + k\cdot [(G\times G)/\Delta(E)]\,,
\end{equation*}
where $E$ denotes the trivial subgroup of $G$. With this, it is straightforward to show that $\gamma_\alpha\cdotG\gamma_\beta=\gamma_{\alpha\beta}$, using again Lemma~\ref{lem Mackey formula}. Thus, $\gamma_\alpha$ belongs to $\Lambda_G$ and $\alpha\mapsto\gamma_\alpha$ induces a group homomorphism $j\colon \Out(H)\to\Lambda_G$, $\alphabar\mapsto \gamma_\alpha$. Since $N_G(H)=H$, $j$ is injective.
\end{proof}

\begin{remark}\label{rem frobenius} 
Assume the notation in Proposition~\ref{prop frobenius}.

\smallskip
(a) If $\Out(H)$ is non-trivial, as for instance for $G=A_4$ and $H=A_3$, then $\gamma_\alpha$ is an orthogonal unit in $B^\Delta_\circ(G,G)$ which is not uniform, whenever $\alpha\notin\Inn(H)$.

\smallskip
(b) Proposition~\ref{prop frobenius} implies hat $\iota(B(G)^\times)\rtimes j(\Out(H))$ is a subgroup of $\Lambda_G$. One might ask if one has equality. The following proposition shows that the answer is positive for $G=A_4$. We do not know the answer for general Frobenius groups $G$. We are confident that the answer is positive whenever the Frobenius kernel $N$ is cyclic. But even in this simple case our computations are too lengthy to be included here.
\end{remark}

\begin{proposition}
Let $G:=A_4$ be the alternating group of degree $4$. Then  $B(G)^\times$, $\Lambda_G$, and $B_\circ^\Delta(G,G)$ are elementary abelian $2$-groups of rank $2$, $3$, and $4$, respectively. 
\end{proposition}

\begin{proof}
We denote by $V$ the subgroup of order $4$ of $G$, set $H:=\langle(1,2,3)\rangle$, $K:=\langle (1,2)(3,4)\rangle$, and $E:=\{1\}$. Then the elements $[G/G]$, $[G/V]$, $[G/H]$, $[G/K]$, $[G/E]$ form the standard basis of $B(G)$. Also, the elements $\Delta(G)$, $\Delta(G,\phi,G)$, $\Delta(V)$, $\Delta(V,\tau,V)$, $\Delta(H)$, $\Delta(H,\alpha,H)$, $\Delta(K)$, $\Delta(E)$ form a set of representatives of the conjugacy classes of $\scrS_{G,G}^\Delta$. Here, $\phi\in\Aut(G)$, $\tau\in\Aut(V)$, and $\alpha\in\Aut(H)$ are given by conjugation with $(1,2)$.

\smallskip
Using the formula in Lemma~\ref{lem Mackey formula} one can check that the element $[(G\times G)/\Delta(G,\phi,G)]$ centralizes the above basis elements and therefore the ring $B^\Delta(G,G)$. This implies that the semidirect product $\Lambda_G\rtimes \Delta_G$ in Lemma~\ref{lem 3}(d) is a direct product. As $B(G)^\times$ is an elementary abelian $2$-group, it suffices to show that $\Lambda_G\smallsetminus \iota(B(G)^\times)$ has precisely $4$ elements and that these elements have order $2$.

\smallskip
Next assume that $\gamma\in\Lambda_G\smallsetminus \iota(B(G)^\times)$ and that its coefficient at $[(G\times G)/\Delta(G)]$ is equal to $1$. We will show that there are precisely two elements $\gamma$ with this property and that they have order $2$. This implies the desired result, since multiplication with $-1$ yields the remaining elements in $\Lambda_G\smallsetminus\iota(B(G)^\times)$. Since $V$ is a maximal subgroup of $G$ and since $N_G(V)=G$, Lemma~\ref{lem 5} implies that the only standard basis element of the form $[(G\times G)/\Delta(V,\beta,V)]$, $\beta\in\Aut(V)$, that can occur in $\gamma$ is the element $[(G\times G)/\Delta(V)]$. On the other hand, since $\gamma$ is not in $\iota(B(G)^\times$, the basis element $[(G\times G)/\Delta(H,\alpha,H)]$ must occur in $\gamma$. This means that $\gamma$ lies in the span of $[(G\times G)/\Delta(G)]$, $[(G\times G)/\Delta(V)]$, $[(G\times G)/\Delta(H)]$, $[(G\times G)/\Delta(H,\alpha,H)]$, $[(G\times G)/\Delta(K)]$ and $[(G\times G)/\Delta(E)]$. Lemma~\ref{lem 5} implies that the coefficient of $[(G\times G)/\Delta(H,\alpha,H)]$ is equal to $1$ or to $-1$. Thus, we have
\begin{align*}
  \gamma= & [(G\times G)/\Delta(G)]+a[(G\times G)/\Delta(V)]+b[(G\times G)/\Delta(H)]+c[(G\times G)/\Delta(H,\alpha,H)]\\
              & +d[(G\times G)/\Delta(K)]+e[(G\times G)/\Delta(E)]
\end{align*}
with $a,b,c,d,e\in\ZZ$ and $c\in\{\pm 1\}$. By Lemma~\ref{lem 1}(a) and Lemma~\ref{lem diag mark formula}  we have 
\begin{align*}
  |C_G(V)| & = |\Phi_{\Delta(V)}(\gamma)| = |\Phi_{\Delta(V)}([(G\times G)/\Delta(G)]+a[(G\times G)/\Delta(V)])| \\
                & = |C_G(V)|\cdot|\Phi_V([G/G]+a[G/V])| = |C_G(V)|\cdot(1+3a)\,.
\end{align*}
Thus, $1=1+3a$ and $a=0$. Since
\begin{equation*}
  \Phi_{\Delta(H,\alpha,H)}(\gamma)= \Phi_{\Delta(H,\alpha,H)}(c[(G\times G)/\Delta(H,\alpha,H)])
      = c\cdot[N_{G\times G}(\Delta(H,\alpha,H)):H]\neq 0\,,
\end{equation*}
The uniqueness part in Lemma~\ref{lem 1}(a) and Lemma~\ref{lem diag mark formula} imply 
\begin{align*}
  0 & =\Phi_{\Delta(H)}(\gamma) = \Phi_{\Delta(H)}([(G\times G)/\Delta(G)]+b[(G\times G)/\Delta(H)]) \\
       & = |C_G(H)|\cdot\Phi_H([G/G]+b[G/H]) = 3(1+b)\,. 
\end{align*}
Thus, $b=-1$. Similarly, we obtain
\begin{equation*}
  |C_G(K)| = |\Phi_{\Delta(K)}(\gamma)| = |C_G(K)|\cdot|\Phi_K([G/G]+d\cdot[G/K])| = |C_G(K)|\cdot |1+2d|\,.
\end{equation*}
Thus, $|1+2d|=1$ and $d\in\{-1, 0\}$. Finally, we have
\begin{equation}\label{eqn last}
  12=|\Phi_{\Delta(E)}(\gamma)| = |12 -48 +48c+72d+144e|\,.
\end{equation}
The two combinations $(c,d)=(1,-1)$ and $(c,d)=(-1,0)$ and Equation~(\ref{eqn last}) imply that $e$ is not an integer, a contradiction. The combinations $(c,d)=(1,0)$ and $(c,d)=(-1,-1)$ lead to $e=0$ and $e=1$, respectively. Thus, there are at most 2 possibilities for $\gamma$ with the properties assumed at the beginning of the section. It is straightforward to check that both possibilities actually lead to elements in $B^\Delta_\circ(G,G)$ of order $2$. Now the result follows. For completeness and in order to have explicit generators of $B_\circ^\Delta(G,G)$, we remark that the elements 
\begin{equation*}
  -[G/G]\quad\text{and}\quad [G/G]-2[G/H]-[G/K]+[G/E]
\end{equation*}
are units of $B(G)$ and therefore generate $B(G)^\times$.
\end{proof}


\end{document}